\documentclass[11pt,a4paper,reqno]{amsart}
\usepackage{layout} 
\def\boxit{$\sqcap\kern-8pt\sqcup$}
\def\littbox{\null\hfill\boxit{}}

\newtheorem{thm}{Theorem}
\newtheorem{lem}{Lemma}

\newtheorem{prop}{Proposition}

\usepackage{fullpage}

\usepackage{amsmath}
\usepackage{amssymb}
\usepackage{amscd}
\usepackage{amsfonts}
\usepackage{amsthm}
\usepackage{graphics}
\usepackage{pgf,xcolor}
\usepackage{color}\input epsf

\usepackage{ulem}
\usepackage{slashbox}
\usepackage{longtable}

\usepackage{dsfont}

\newcommand{\N}{\mathbb{N}}
\newcommand{\Z}{\mathbb{Z}}

\newcommand{\R}{\mathbb{R}}

\newcommand{\m}{\mathfrak{m}}

\begin{document}
\date{December 15, 2012}
\title[On the M\"{o}bius function of the locally finite poset associated with a numerical semigroup]{On the M\"{o}bius function of the locally finite poset associated with a numerical semigroup}

\author[Jonathan CHAPPELON]{Jonathan CHAPPELON}
\address{Institut de Math\'ematiques et de Mod\'elisation de Montpellier, Universit\'e Montpellier 2, 
Place Eug\`ene Bataillon, 34095 Montpellier}
\email{jonathan.chappelon@math.univ-montp2.fr}
\urladdr{http://www.math.univ-montp2.fr/~chappelon/}

\author[Jorge Luis Ram\'{i}rez Alfons\'{i}n]{Jorge Luis Ram\'{i}rez Alfons\'{i}n}
\address{Institut de Math\'ematiques et de Mod\'elisation de Montpellier, Universit\'e Montpellier 2, 
Place Eug\`ene Bataillon, 34095 Montpellier}
\email{jramirez@math.univ-montp2.fr}
\urladdr{http://www.math.univ-montp2.fr/~ramirez/}

\begin{abstract}
Let $S$ be a numerical semigroup and let $(\Z,\leq_S)$ be the (locally finite) poset induced by $S$ on the set of integers $\Z$ defined by $x \leq_S y$ if and only if $y-x\in S$ for all integers $x$ and $y$. In this paper, we investigate the {\it M\"{o}bius function} associated to $(\Z,\leq_S)$ when $S$ is an {\it arithmetic} semigroup. 
\end{abstract}
\maketitle
%\thanks{}
\smallskip


\section{Introduction} \label{intro} 

The {\it M\"{o}bius function} is an important concept associated to ({\it locally finite}) posets.  M\"{o}bius function can be considered as a generalization of the classical M\"{o}bius arithmetic function on the integers (given by the  M\"{o}bius function of the poset obtained from the  positive integers partially ordered by the divisibility). 
M\"{o}bius function  has been extremely useful to investigate many different problems.
For instance, the {\it inclusion-exclusion principle} can be retrieved by considering the set of all subsets of a finite set partially ordered by inclusion. We refer the reader to \cite{Rota} for a large number of applications of the M\"{o}bius function. 
\smallskip

In this paper, we investigate the M\"{o}bius function associated to posets arising naturally from numerical semigroups as follows. Let $a_1,a_2,\ldots,a_n$ be $n\geq1$ relatively prime positive integers and let $S=\left\langle a_1,a_2,\ldots,a_n\right\rangle$ denote the {\it numerical semigroup} generated by $a_1,a_2,\ldots,a_n$, that is,
$$
S = \left\langle a_1,a_2,\ldots,a_n \right\rangle = \left\{ x_1a_1 + x_2a_2 + \cdots + x_na_n \ \middle|\ x_1,x_2,\ldots,x_n\in\N \right\}.
$$
Throughout this paper, we consider the structure of the poset induced by $S$ on the set of integers $\Z$, whose partial order $\leq_S$ is defined by
$$
x \leq_S y\quad \Longleftrightarrow\quad y-x\in S,
$$
for all integers $x$ and $y$. This (locally finite) poset will be denoted by $(\Z,\leq_S)$. 
\smallskip

We denote by $\mu_S$ the  M\"{o}bius function associated to $(\Z,\leq_S)$. As far as we are aware, the only known result concerning $\mu_S$ is  an old theorem due to Deddens \cite{Deddens} that determines the value of $\mu_S$ when $S$ has exactly two generators.
Here, we shall introduce and develop a new approach to investigate $\mu_S$ when $S$ is  an {\it arithmetic semigroup},  that is, when $S=\langle a,a+d,\dots, a+kd\rangle$ for some integers $a,d$ and $k\le a-1$. 
\smallskip

This is a self-contained paper and it is organized as follows. In the next section, we review some classic notions of the M\"{o}bius function and present some results  needed for the rest of the paper. In Section \ref{Dedd}, we give a new direct proof of Deddens' result, shorter than the original one (based on a recursive case-by-case analysis). In Section \ref{prelim}, we discuss results about arithmetic semigroups, in particular, we prove the existence of  {\it unique} representations. The latter is a key result that will be used, in Section \ref{recurs},  to give a recursive formula for $\mu_S$ when $S=\langle a,a+d,\dots, a+kd\rangle$. Finally, in Section \ref{pair}, we propose an explicit formula for $\mu_S$ (based on the {\it multiplicity function} of a multiset)  in the case when $k=2$ and $a$ is even.

Background information on numerical semigroups can be found in the books \cite{R1,RS}.
\section{M\"{o}bius function}\label{basic}

Let $(P,\leq )$ be a partially ordered set, or {\it poset} for short. The {\it strict partial order} $<_P$ is the reduction of $\leq_P$ given by, $a <_P b$ if and only if $a\leq_P b$ and $a\neq b$. For any $a$ and $b$ in the poset $P$, the {\it segments} between $a$ and $b$ are defined by
$$
\begin{array}{l@{\ =\ }l@{\quad\quad}l@{\ =\ }l}
{\left[a,b\right]}_{P} & \left\{ c\in P\ \middle|\ a\leq_P c\leq_P b \right\}, & {\left]a,b\right]}_{P} & \left\{ c\in P\ \middle|\ a <_P c\leq_P b \right\}, \\[2ex]
{\left[a,b\right[}_{P} & \left\{ c\in P\ \middle|\ a\leq_P c <_P b \right\}, & {\left]a,b\right[}_{P} & \left\{ c\in P\ \middle|\ a <_P c <_P b \right\}. \\
\end{array}
$$
A poset is said to be {\it locally finite} if every segment has finite cardinality. In this paper, we only consider locally finite posets.
\par Let $a$ and $b$ be elements of the poset $P$. A {\it chain} of length $l\geq 0$ between $a$ and $b$ is a subset of ${\left[a,b\right]}_{P}$ containing $a$ and $b$, with cardinality $l+1$ and totally ordered by $<$, that is $\left\{ a_0,a_1,\ldots,a_l\right\}\subset{\left[a,b\right]}_{P}$ such that
$$
a=a_0 <_P a_1 <_P a_2 <_P \cdots <_P a_{l-1} <_P a_l = b.
$$
For any nonnegative integer $l$, we denote by $C_l(a,b)$ the set of all chains of length $l$ between $a$ and $b$. The cardinality of $C_l(a,b)$ is denoted by $c_l(a,b)$. This number is always finite because the poset $P$ is supposed to be locally finite. For instance, the number of chains $c_2(2,12)$, where the poset is the set $\N$ partially ordered by divisibility, is equal to $2$. Indeed, there are exactly $2$ chains of length $2$ between $2$ and $12$ in ${\left[2,12\right]}_{\N}=\left\{2,4,6,12\right\}$, which are $\left\{2,4,12\right\}$ and $\left\{2,6,12\right\}$.

For any locally finite poset $P$, the {\it M\"{o}bius function} $\mu_{P}$ is the integer-valued function on $P\times P$ defined by
\begin{equation}\label{eq1}
\mu_{P}(a,b) = \sum_{l\geq 0}{{(-1)}^{l}c_l(a,b)},
\end{equation}
for all elements $a$ and $b$ of the poset $P$. One can remark that this sum is always finite because, for $a$ and $b$ given, there exists a maximal length of a possible chain between $a$ and $b$ since the segment ${\left[a,b\right]}_{P}$ has finite cardinality.

The concept of M\"{o}bius function for a locally finite poset $(P,\le )$ was introduced by Rota in \cite{Rota} as the inverse of the zeta function in the incidence algebra of a locally finite poset. Let us see this with more detail. Consider the set $\mathcal{I}(P)$ of all real-valued functions $f : P\times P \longrightarrow \R$ for which $f(a,b)=0$ if $a{\not\leq}_{P}b$. The sum $+$ and the multiplication by scalars $.$ are defined as usual in $\mathcal{I}(P)$. The product of two functions $f$ and $g$ in $\mathcal{I}(P)$ is defined by
$$
(f\times g)(a,b) = \sum_{c\in{\left[a,b\right]}_{P}}f(a,c)g(c,b),
$$
for all $(a,b)\in P\times P$. Then $\left(\mathcal{I}(P),+,.,\times\right)$ appears as an associative algebra over $\R$. This is the {\it incidence algebra} of $P$. The {\it Kronecker delta function} $\delta\in\mathcal{I}(P)$, defined by
$$
\delta(a,b) = \left\{\begin{array}{ll}
1 & \text{if}\ a=b,\\
0 & \text{otherwise},
\end{array}\right.
$$
for all $(a,b)\in P\times P$, is the identity element of $\mathcal{I}(P)$. The {\it zeta function} ${\zeta}_{P}\in\mathcal{I}(P)$ is defined by
$$
{\zeta}_{P}(a,b) = \left\{\begin{array}{ll}
1 & \text{if}\ a\leq_{P}b,\\
0 & \text{otherwise},
\end{array}\right.
$$
for all $(a,b)\in P\times P$. 

Rota \cite{Rota} proved that the zeta function ${\zeta}_{P}$ (called the {\it inverse function}) is invertible in $\mathcal{I}(P)$ and  showed that $\mu_{P}$ is recursively defined as follows: for all $(a,b)\in P\times P$, by
\begin{equation}\label{eq2}
\mu_{P}(a,a) = 1 \quad \text{and} \quad \mu_{P}(a,b) = -\sum_{c\in{\left[a,b\right[}_{P}}\mu_{P}(a,c) \quad \text{if}\ a<_P b.
\end{equation}
Let us see that both definitions of $\mu_S$ given by~(\ref{eq1}) and by~(\ref{eq2})  are equivalent. For, let $a$ and $b$ 
be two elements of the locally finite poset $P$ such that $a<_{P}b$. Then,
\begin{equation}\label{eq3}
c_l(a,b) = \sum_{c\in{\left[a,b\right[}_{P}}c_{l-1}(a,c) = \sum_{c\in{\left]a,b\right]}_{P}}c_{l-1}(c,b),
\end{equation}
for all positive integers $l$. Indeed,
every chain $\left\{a_0,a_1,\ldots,a_l\right\}\in C_l(a,b)$ can be seen like an extension of a chain of $C_{l-1}(a,a_{l-1})$ or of $C_{l-1}(a_1,b)$.

Obviously, the identity $\mu_S(a,a)=1$ directly comes from (\ref{eq1}) since $c_0(a,a)=1$ and $c_l(a,a)=0$ for all $l\geq 1$. By combining (\ref{eq3}) and (\ref{eq1}), for all $a<_P b$, we obtain that
$$
\mu_P(a,b) = \sum_{l\geq 0}{{(-1)}^{l}c_l(a,b)} = c_0(a,b) + \sum_{l\geq 1}{{(-1)}^{l}\sum_{c\in{\left[a,b\right[}_{P}}c_{l-1}(a,c)}.
$$
Finally, since $a\neq b$, it follows that $c_0(a,b)=0$ and thus
$$
\mu_P(a,b) = \sum_{c\in{\left[a,b\right[}_{P}}\sum_{l\geq 0}{{(-1)}^{l+1}c_{l}(a,c)} = -\sum_{c\in{\left[a,b\right[}_{P}}\mu_{P}(a,c).
$$
Similarly, using the second identity of (\ref{eq3}), we can also prove that, whenever $a<_P b$, we have
$$
\mu_P(a,b)=-\sum_{c\in{\left]a,b\right]}_{P}}\mu_P(c,b).
$$
Therefore the two definitions of the M\"{o}bius function (for a locally finite posets) are the same.
All the results presented in this paper are derived from the recursive formula presented in (\ref{eq2}). 

\subsection{Poset of integers induced by a numerical semigroup}

Let $S$ be a numerical semigroup and  $(\Z,\leq_S)$ its associated poset. Observe that
$(\Z,\leq_S)$ is a locally finite poset since $\left| {\left[x,y\right]}_{(\Z,\leq_S)} \right| \leq y-x$, for all $x,y\in\Z$.
It is easy to see that $\mu_S$ can be considered as a univariable function of $\Z$. Indeed,
for all $x,y\in\Z $ and for all $p\ge 0$, we have

\begin{equation}\label{c1}
 c_l(x,y)=c_l(0,y-x).
\end{equation}

The above follows since the set $C_l(x,y)$ is in bijection with $C_l(0,y-x)$. Indeed the map that assigns the chain $\left\{x_0,x_1,\ldots,x_l\right\}\in C_l(x,y)$ to the chain $\left\{0,x_1-x_0,\ldots,x_l-x_0\right\}\in C_l(0,y-x)$ is clearly a bijection.
Thus, by definition of $\mu_S$ and equality \eqref{c1} we obtain

$$\mu_S(x,y)=\mu_S(0,y-x)$$

for all $x,y\in\Z$.

In the sequel of this paper we shall only consider the reduced M\"obius function $\mu_S : \Z \longrightarrow \Z$ defined by
$$
\mu_S(x) = \mu_S(0,x),\quad \text{for\ all}\ x\in\Z.
$$

This recursive formula given by (\ref{eq2}) can be more easily presented when the locally finite poset is $(\Z,\leq_S)$.

\begin{prop}\label{in} Let $S$ be a numerical semigroup and let $x\in\Z\setminus\{0\}$. Then,
$$
\mu_S(x) = -\sum_{y\in S\setminus\{0\}}\mu_S(x-y)
\quad
\Longleftrightarrow
\quad
\sum_{y\in S}\mu_S(x-y) = 0.
$$
\end{prop}

\begin{proof}
From (\ref{eq2}), we deduce that
$$
\mu_S(x) = -\sum_{y\in{\left[0,x\right[}_{(\Z,\leq_S)}}\mu_S(y) = -\sum_{\stackrel{y\in S}{x-y\in S\setminus\{0\}}}\mu_S(y) = 
 -\sum_{\stackrel{x-y\in S}{y\in S\setminus\{0\}}}\mu_S(x-y).
$$
The result follows since, by definition of $\mu_S$, $\mu_S(x-y)=0$ unless $x-y\in S$.
\end{proof}

%%%%%%%%%%%%%%%%%%%%%%%%%%%%%%%

\section{Deddens' result : new proof}\label{Dedd}

In \cite{Deddens}, Deddens proved the following. 

\begin{thm}\cite{Deddens}\label{thmdedd}
Let $a$ and $b$ be two relatively positive integers and let $S=\left\langle a,b\right\rangle$. Then, for all $x\in\Z$, we have
$$
\mu_S(x) = \left\{\begin{array}{rl}
 1 & \text{if}\ x\geq0\ \text{and}\ x\equiv 0\ \text{or}\ a+b \pmod{ab}, \\
 -1 & \text{if}\ x\geq0\ \text{and}\ x\equiv a\ \text{or}\ b \pmod{ab}, \\
 0 & \text{otherwise}.
\end{array}\right.
$$
\end{thm}

Dedden's proof was based on a recursive argument and a case-by-case analysis.
We may give the following direct proof of Theorem \ref{thmdedd}.

{\it Proof of Theorem \ref{thmdedd}.} We shall prove that 
\begin{equation}\label{r0}
\mu_S(x)=\mu_S(x-ab)
\end{equation}

for every $x\in\Z\setminus\{0,a,b,a+b\}$. The result then follows since 
$\mu_S(x)=0$ for all $x<0$, $\mu_S(0)=1$, $\mu_S(a)=\mu_S(b)=-1$ and $\mu_S(a+b)=c_2(0,a+b)-c_1(0,a+b)=2-1=1$.
\smallskip

Let us prove then equality \eqref{r0}. Let $S=\left\langle a,b\right\rangle = \left\{m_aa+m_bb\ \middle|\ m_a,m_b\in\N \right\}$
and let  $x\in \Z\setminus\{0\}$. By Proposition \ref{in} we already know that
$$
\mu_S(x) = -\sum_{y\in S\setminus\{0\}}\mu_S(x-y) = -\sum_{\stackrel{y\in S\setminus\{0\}}{y-a\in S}}\mu_S(x-y) -\sum_{\stackrel{y\in S\setminus\{0\}}{y-a\notin S}}\mu_S(x-y).
$$
Since
$$
\sum_{\stackrel{y\in S\setminus\{0\}}{y-a\in S}}\mu_S(x-y) = \sum_{z\in S}\mu_S((x-a)-z) = 0,\quad \text{for}\ x-a\neq 0,
$$
then 
$$
\mu_S(x) = -\sum_{\stackrel{y\in S\setminus\{0\}}{y-a\notin S}}\mu_S(x-y),\quad \text{for}\ x\in \Z\setminus\{0,a\}.
$$

Moreover since  $\left\{y\in S\setminus\{0\}\ \middle|\ y-a\notin S\right\}=\left\{ m_b b\ \middle|\ m_b\in\{1,2,\ldots,a-1\}\right\}$
then
\begin{equation}\label{r1}
\mu_S(x) = -\sum_{m_b=1}^{a-1}\mu_S(x-m_bb),\quad \text{for}\ x\in \Z\setminus\{0,a\}.
\end{equation}

By applying \eqref{r1} for $x-b\in \Z\setminus\{0,a\}$, that is, $x\in \Z\setminus\{b,a+b\}$ we obtain that

\begin{equation}\label{r2}
\mu_S(x-b) = -\sum_{m_b=2}^{a}\mu_S(x-m_bb),\quad \text{for}\ x-b\in \Z\setminus\{0,a\}.
\end{equation}

By combining \eqref{r1} and \eqref{r2}, for $x\in \Z\setminus\{0,a,b,a+b\}$, we obtain that

$$
\mu_S(x)
\begin{array}[t]{l}
 = -\displaystyle\sum_{m_b=1}^{a-1}\mu_S(x-m_bb) = -\mu_S(x-b)-\sum_{m_b=2}^{a-1}\mu_S(x-m_bb) \\
 = \displaystyle\sum_{m_b=2}^{a}\mu_S(x-m_bb)-\sum_{m_b=2}^{a-1}\mu_S(x-m_bb) \\[3ex]
 = \mu_S(x-ab)
\end{array}
$$

as desired.
\littbox

%%%%%%%%%%%%%%%%%%%%%%%%%%%

\section{Arithmetic semigroups : preliminary results}\label{prelim}

Let $S$ be a numerical semigroup.
The {\it Ap\'{e}ry set} of $S$ with respect with $m\in S$ is defined as
$$
Ap(S;m) = \left\{ x\in S\ \middle|\ x-m\notin S \right\}.
$$
It is known that $Ap(S;m)$ constitutes a complete set a of residues $\bmod m$.

Roberts \cite{Ro1} has proved that if $S=\left\langle a,a+d,\ldots,a+kd\right\rangle$ with $\gcd(a,d)=1$ and $k\in\{1,2,\ldots,a-1\}$ then
\begin{equation}\label{rob}
Ap(S;a) = \left\{ \left\lceil \frac{i}{k} \right\rceil a + id\ \middle|\ i\in\{0,1,\ldots,a-1\} \right\}.
\end{equation}

The following result gives a unique representation of elements in arithmetic semigroups.

\begin{lem}\label{rep}
Let $x\in S=\left\langle a,a+d,\ldots,a+kd\right\rangle$ with $2\le k\le a-1$. Then, there exists a unique triplet $(x_0,x_i,x_k)\in \N\times \{0,1\}\times \{0,\dots,\lceil\frac{a}{k}\rceil\}$ such that
$$
x = x_0 a + x_i (a+id) + x_k (a+kd)
$$
for some $1\leq i\leq k-1$ with $i x_i + k x_k  < a$. 
\end{lem}

\begin{proof} Let $x\in S$ and $x_0,\ldots,x_k\in\N$ such that
\begin{equation}\label{eqq}
x=x_0a+x_1(a+d)+x_2(a+2d)+\cdots+x_k(a+kd).
\end{equation}

\underline{{\it Existence}}: Let $j_1$ and $j_2$ be two integers such that $0\le j_1,j_2\le k$.
We notice that $(a+j_1d)+(a+j_2d)$ can be expressed as either

$$\hbox{$(a+j_1d) + (a+j_2d) = a + (a+(j_1+j_2)d)$, for $0\le j_1+j_2\le k$}$$
or

$$\hbox{$(a+j_1d) + (a+j_2d) = (a+kd) + (a+(j_1+j_2-k)d)$, for $k\le j_1+j_2\le 2k$.}$$

So, by repeatedly adding consecutive terms from the expression of $x$ in \eqref{eqq}, we obtain that 
 there exists a triplet $(x_0,x_i,x_k)$ such that
$$
x = x_0a + x_i(a+id) + x_k(a+kd),
$$
with $1\leq i\leq k-1$ and $x_i\in\{0,1\}$. Moreover, we may suppose that $0\le x_k \le \left\lfloor a/k \right\rfloor$. Otherwise, we use the following equality
$$
\left( \left\lfloor \frac{a}{k} \right\rfloor + 1 \right)(a+kd) \begin{array}[t]{l}
= \displaystyle\left( \left\lfloor \frac{a}{k} \right\rfloor + 1 \right)a + \left( \left\lfloor \frac{a}{k} \right\rfloor k + k \right)d \\[2ex]
= \displaystyle\left( \left\lfloor \frac{a}{k} \right\rfloor + d + 1 \right)a + \left( \left\lfloor \frac{a}{k} \right\rfloor k + k - a \right)d \\[2ex]
= \displaystyle\left( \left\lfloor \frac{a}{k} \right\rfloor + d \right)a + \left( a + \left( \left\lfloor \frac{a}{k} \right\rfloor k + k - a \right)d \right),
\end{array}
$$
where $1\le \left\lfloor a/k \right\rfloor k + k - a \le k$.

Finally, if $i x_i+k x_k\ge a$ then we consider the following representation
 $$
x \begin{array}[t]{l}
= x_0a + x_i(a+id) + x_k(a+kd) \\[2ex]
= \displaystyle\left( x_0+x_i+ x_k \right)a + \left( i x_i+ k x_k\right)d \\[2ex]
= \displaystyle\left( x_0+ x_i+ x_k+d \right)a + \left( i x_i+ k x_k - a \right)d \\[2ex]
= \displaystyle\left( x_0+ x_i + x_k+ d - 1 \right)a + \left( a + \left( i x_i+ k x_k- a \right)d \right),
\end{array}
$$
where $0 \le i x_i+k x_k - a\le ix_i+k\left\lfloor{a}/{k}\right\rfloor - a \le k -1 + \left\lfloor{a}/{k}\right\rfloor k - a\le k-1$. 
Obtaining the desired decomposition.
\medskip

\underline{{\it Uniqueness}}: Let us suppose that there exist two triplets of non-negative integers $(x_0,x_i,x_k)$ and $(y_0,y_j,y_k)$ such that
$$
x_0 a + x_i (a+id) + x_k (a+kd) = y_0 a + y_j(a+jd) + y_k (a+kd),
$$
with $1\le i,j\le k-1$, $x_i,y_j\in\{0,1\}$, $i x_i+k x_k  < a$ and $j y_j +k y_k  < a$. It follows that
$$
(i x_i + k x_k )d\equiv (j y_j + k y_k )d \pmod{a},
$$
and since $\gcd(a,d)=1$ then
$$
i x_i + x_k k \equiv j y_j + k y_k  \pmod{a}
$$
Moreover, since $ i x_i + k x_k  < a$ and $ j y_j+ k y_k < a$, then
$$
i x_i + k x_k = j y_j+ k y_k \text{ or equivalently } i x_i-j y_j=k(y_k-x_k) .
$$
We have four cases.
\smallskip

Case 1) if $x_i=0$ and $y_j=1$ then $-j$ would be a multiple of $k$ which is impossible since $-1\ge -j\ge -k$
\smallskip

Case 2) if $x_i=1$ and $y_j=0$ then $i$ would be a multiple of $k$ which is impossible since $1\le i\le k$
\smallskip

Case 3) if $x_i=y_j=1$ then $i-j$ would be a multiple of $k$ but  since $-k+2\le i-j\le k-2$ then $i-j=0$ implying that
$x_k=y_k$ and thus $x_0=y_0$.
\smallskip

Case 4) if $x_i=y_j=0$ then $k(x_k-y_k)=0$ and since $k\ge 1$ then $x_k=y_k$ and thus $x_0=y_0$.
\end{proof}

Let $(x_0,x_i,x_k)\in \N\times \{0,1\}\times \{0,\dots ,\left\lfloor a/k \right\rfloor\}$ with $1\leq i\leq k-1$ and $i x_i + k x_k  < a$.
We shall denote  by $[x_0,x_i,x_k]$ the element in $S$ given by the representation of Lemma \ref{rep}. 
%%%%%%%%%%%%%%%%%%%%%%%%%%%%%%%%%%%

\section{Recursive formula}\label{recurs}

We shall now present a recursive formula for $\mu_S$ when $S=\left\langle a,a+d,\ldots,a+kd\right\rangle$. 
The following key remark led us to guess such recursion. If $x=m_aa+m_dd$ such that $m_a\geq 0$ and $0\leq m_d\leq a-1$ then,
$$
x\in S\quad \Longleftrightarrow\quad m_a\geq \left\lceil\frac{m_d}{k} \right\rceil.
$$

\begin{thm}\label{thm} Let $S=\left\langle a,a+d,\ldots,a+kd\right\rangle$ with $\gcd(a,d)=1$ and let $a=qk+r$ with $0\le r<k$. Let $x\in\Z\setminus\{0,a,a+kd,a+(a+kd)\}$, then

$$\mu_S(x)=\left\{\begin{array}{ll}
\mu_S(x-q(a+kd)) + \displaystyle\sum_{i=1}^{k-1}\mu_S(x-(a+id)-q(a+kd)) \\- \mu_S(x-(a+id)) & \text{ if } r=0,\\
\mu_S(x-(q+1)(a+kd)) 
+ \displaystyle\sum\limits_{i=r}^{k-1}\mu_S(x-(a+id)-q(a+kd)) \\ - \displaystyle\sum_{i=1}^{k-1}\mu_S(x-(a+id)) & \text{ if } r=1,\\
\mu_S(x-(q+1)(a+kd)) + \displaystyle\sum_{i=1}^{r-1}\mu_S(x-(a+id)-(q+1)(a+kd)) \\
+ \displaystyle\sum_{i=r}^{k-1}\mu_S(x-(a+id)-q(a+kd)) - \displaystyle\sum_{i=1}^{k-1}\mu_S(x-(a+id)) & \text{ if } r\ge 2.\\
\end{array}\right.$$
\end{thm}

\begin{proof} Let $x\in \Z\setminus\{0\}$. As for the proof of Theorem \ref{thmdedd}, we have
$$
\mu_S(x)
\begin{array}[t]{l}
 = -\hspace{-10pt}\displaystyle\sum_{y\in S\setminus\{0\}}\mu_S(x-y) \\[4ex]
 = -\hspace{-8pt}\displaystyle\sum_{\stackrel{y\in S\setminus\{0\}}{y-a\in S}}\mu_S(x-y) -\hspace{-8pt}\sum_{\stackrel{y\in S\setminus\{0\}}{y-a\notin S}}\mu_S(x-y) \\[6ex]
 = -\hspace{-2pt}\displaystyle\sum_{z\in S}\mu_S((x-a)-z) -\hspace{-8pt}\sum_{\stackrel{y\in S\setminus\{0\}}{y-a\notin S}}\mu_S(x-y) \\
 = -\hspace{-8pt}\displaystyle\sum_{\stackrel{y\in S\setminus\{0\}}{y-a\notin S}}\mu_S(x-y)\quad\text{if}\quad x-a\neq 0.
\end{array} 
$$

Let us now determine the set
$$
\left\{y\in S\setminus\{0\}\ \middle|\ y-a\notin S\right\} = Ap(S;a)\setminus\{0\}.
$$

For, we consider the set $Ap(S;a)$ given by \eqref{rob} in function of the unique representation of Lemma \ref{rep}. 
We have three cases.

Case a) If $r=0$ then
$$
Ap(S,a)\setminus\{0\} = \begin{array}[t]{l}
\displaystyle\left\{y_k(a+kd)\ \middle|\ y_k\in\{1,\ldots,q-1\}\right\} \\
\displaystyle\bigcup \left\{ (a+id)+y_k(a+kd)\ \middle|\ \begin{array}{l} y_k\in\{0,\ldots,q-1\}\\ i\in\{1,\ldots,k-1\} \\ \end{array} \right\} \\
\end{array}
$$
Case b) If $r=1$ then
$$
Ap(S,a)\setminus\{0\} = \begin{array}[t]{l}
\displaystyle\left\{y_k(a+kd)\ \middle|\ y_k\in\{1,\ldots,q\}\right\} \\
\displaystyle\bigcup \left\{ (a+id)+y_k(a+kd)\ \middle|\ \begin{array}{l} y_k\in\{0,\ldots,q-1\}\\ i\in\{r,\ldots,k-1\} \\ \end{array} \right\} \\
\end{array}
$$
 Case c) If $r\geq 2$ then
$$
Ap(S,a)\setminus\{0\} = \begin{array}[t]{l}
\displaystyle\left\{y_k(a+kd)\ \middle|\ y_k\in\{1,\ldots,q\}\right\} \\
\displaystyle\bigcup \left\{ (a+id)+y_k(a+kd)\ \middle|\ \begin{array}{l} y_k\in\{0,\ldots,q\}\\ i\in\{1,\ldots,r-1\} \\ \end{array} \right\} \\
\displaystyle\bigcup \left\{ (a+id)+y_k(a+kd)\ \middle|\ \begin{array}{l} y_k\in\{0,\ldots,q-1\}\\ i\in\{r,\ldots,k-1\} \\ \end{array} \right\} \\
\end{array}
$$

Suppose that $r=0$, i.e. $a=qk$. For $x\in \Z\setminus\{0,a\}$, we have
\begin{equation}\label{mmu1}
\mu_S(x) = -\sum_{y_k=1}^{q-1}\mu_S(x-y_k(a+kd))-\sum_{i=1}^{k-1}\sum_{y_k=0}^{q-1}\mu_S(x-(a+id)-y_k(a+kd)).
\end{equation}

By applying \eqref{mmu1} to $x-(a+kd)\in \Z\setminus\{0,a\}$, that is, $x\in\Z\setminus\{a+kd,a+(a+kd)\}$ we obtain 

\begin{equation}\label{mmu2}
\mu_S(x-(a+kd)) = -\sum_{y_k=2}^{q}\mu_S(x-y_k(a+kd))-\sum_{i=1}^{k-1}\sum_{y_k=1}^{q}\mu_S(x-(a+id)-y_k(a+kd)).
\end{equation}

By combining \eqref{mmu1} and \eqref{mmu2} for $x\in \Z\setminus\{0,a,a+kd,a+(a+kd)\}$, we obtain 
$$
\mu_S(x) = \mu_S(x-q(a+kd)) + \sum_{i=1}^{k-1}\mu_S(x-(a+id)-q(a+kd)) - \mu_S(x-(a+id)).
$$
The cases when $r\geq 1$ are similar to Case a (and it is left to the reader as an exercise).
\end{proof}

%%%%%%%%%%%%%%%%%%%%%%%%%%%

\section{Case $\langle 2q, 2q+d,2q+2d\rangle$}\label{pair}

The {\it multiplicity function} of a multiset $A$ of $\N$ is the function
$$
\m_A : \N \longrightarrow \N
$$
which assigns to each element $x\in\N$ its multiplicity, that is, the number of times that $x$ appears in the multiset $A$.

Let $a=2q$ and $d\in\N^*$ such that $\gcd(a,d)=\gcd(q,d)=1$. For each $i\in\{-1,0,1\}$, we consider the following multisets.
$$
\begin{array}{l}
A_i = \left\{ m(q+d)+i \ \middle|\ m\in\N \right\}, \\[2ex]
B_i = \left\{ m(q+d)-nd+i \ \middle|\ m\in\N\setminus\{0,1\} , n\in\{1,2,\ldots,\left\lfloor m/2\right\rfloor\} \right\}, \\[2ex]
C_i = A_i\bigcup B_i.
\end{array}
$$
As we mentioned above, given a triple $(x_0,x_1,x_2)\in\N\times\{0,1\}\times\{0,\ldots,q-1\}$, we denote by $[x_0,x_1,x_2]$ the element in $S$ given by the representation in  Lemma \ref{rep}.
In the sequel of this section, we shall consider this representation for all $x_0\in\Z$, i.e.,
$$
[x_0,x_1,x_2] = x_0a + x_1(a+d) + x_2(a+2d)
$$
for all $(x_0,x_1,x_2)\in\Z\times\{0,1\}\times\{0,\ldots,q-1\}$. In this case, it is clear that
$$
\text{if } (x_0,x_1,x_2)\in\Z\times\{0,1\}\times\left\{0,\ldots,q-1\right\}\text{ then } [x_0,x_1,x_2]\in S\Longleftrightarrow x_0\in\N.
$$

The latter will be used in the proofs below.

\begin{thm}\label{thm1}
Let $S=<2q,2q+d,2q+2d>$. Let $(x_0,x_1,x_2)\in \Z\times \{0,1\}\times \{0,\dots,q-1\}$.
Then,
$$
\mu_S ([x_0,x_1,x_2]) = \left\{\begin{array}{ll}
{(-1)}^{x_1} \left( \m_{A_0}-\m_{A_1}+2\m_{B_0}-\m_{B_{-1}}-\m_{B_1}  \right)(x_0) & \text{ if } x_2=0,\\
{(-1)}^{x_1} \left( 2\m_{C_0}-\m_{C_{-1}}-\m_{C_1}  \right)(x_0-x_2) & \text { if } x_2\geq 1.
\end{array}\right.
$$
\end{thm}

We notice that if $x_0-x_2$ is a constant then we should have the same value for $\mu_S([x_0,0,x_2])$. The latter is illustrated 
by the first values of $\mu_S([x_0,0,x_2])$, listed in Table~\ref{table} given at the end of the section, for the case when $a=22$ and $d=5$. 
Indeed, we can see appearing diagonals (corresponding to $x_0-x_2$ constant) with the same value.
\medskip

Before proving Theorem \ref{thm1}, we need two lemmas and the following refinement of Theorem \ref{thm} when $k=2$ and $a$ even.

\begin{prop}\label{prop1}
Let $(x_0,x_1,x_2)\in \Z\times \{0,1\}\times \{0,\dots,q-1\}$
with $(x_0,x_2)\notin\{0,1\}\times\{0,1\}$. Then,
$$
\mu_S\left([x_0,x_1,0]\right) = \begin{array}[t]{l}
\mu_S\left([x_0-(q+d),x_1,0]\right) \\[1.5ex]
+ \mu_S\left([x_0-(q+d)-1,x_1,q-1]\right) \\[1.5ex]
- \mu_S\left([x_0-2(q+d)-1,x_1,q-1]\right) \\
\end{array}
$$
and
$$
\mu_S\left([x_0,x_1,x_2]\right) = \begin{array}[t]{l}
\mu_S\left([x_0-(q+d),x_1,x_2]\right) \\[1.5ex]
+ \mu_S\left([x_0-1,x_1,x_2-1]\right) \\[1.5ex]
- \mu_S\left([x_0-(q+d)-1,x_1,x_2-1]\right) \\
\end{array}
$$
when  $x_2\geq1$.
\end{prop}

\begin{proof}
From Theorem~\ref{thm}, we have

\begin{equation}\label{mu1}
\mu_S\left( [x_0,x_1,x_2]\right) = \begin{array}[t]{l}
\mu_S\left( [x_0,x_1,x_2] - [q+d,0,0] \right) + \mu_S\left( [x_0,x_1,x_2] - [q+d,1,0] \right) \\
 - \mu_S\left( [x_0,x_1,x_2] - [0,1,0] \right).
\end{array}
\end{equation}

Notice that $q(a+2d)=(q+d)a$ and $a+(a+2d)=2(a+d)$. 

Case a) If $x_1=0$ and $x_2=0$ then, from \eqref{mu1} we obtain

\begin{equation}\label{mu2}
\mu_S\left( [x_0,0,0]\right) \begin{array}[t]{l}
= \begin{array}[t]{l}
\mu_S\left( [x_0,0,0] - [q+d,0,0] \right) + \mu_S\left( [x_0,0,0] - [q+d,1,0] \right) \\
 - \mu_S\left( [x_0,0,0] - [0,1,0] \right).
\end{array} \\
= \begin{array}[t]{l}
\mu_S\left( [x_0-(q+d),0,0] \right) + \mu_S\left( [x_0-2(q+d)-1,1,q-1] \right) \\
 - \mu_S\left( [x_0-(q+d)-1,1,q-1] \right).
\end{array} \\
\end{array}
\end{equation}

By applying the  recursive equality \eqref{mu1} to $\mu_S\left( [x_0-(q+d)-1,1,q-1] \right)$, we obtain 
\begin{equation}\label{mu3}
\mu_S\left( [x_0-(q+d)-1,1,q-1]\right) = \begin{array}[t]{l}
   \displaystyle\mu_S\left( [x_0-2(q+d)-1,1,q-1] \right) \\[2ex]
 + \displaystyle\mu_S\left( [x_0-2(q+d)-1,0,q-1]\right) \\[2ex]
 - \displaystyle\mu_S\left( [x_0-(q+d)-1,0,q-1]\right).
\end{array}
\end{equation}

Finally, by combining equations \eqref{mu2} and \eqref{mu3} we have
$$
\mu_S\left([x_0,0,0]\right) = \begin{array}[t]{l}
\mu_S\left([x_0-(q+d),0,0]\right) \\[1.5ex]
+ \mu_S\left([x_0-(q+d)-1,0,q-1]\right) \\[1.5ex]
- \mu_S\left([x_0-2(q+d)-1,0,q-1]\right). \\
\end{array}
$$
Case b) If  $x_1=0$ and $x_2\ge1$ then, from \eqref{mu1} we obtain

\begin{equation}\label{mu4}
\mu_S\left( [x_0,0,x_2]\right) \begin{array}[t]{l}
= \begin{array}[t]{l}
\mu_S\left( [x_0,0,x_2] - x[q+d,0,0] \right) + \mu_S\left( [x_0,0,x_2] - [q+d,1,0] \right) \\
 - \mu_S\left( [x_0,0,x_2] - x[0,1,0] \right)
\end{array} \\
= \begin{array}[t]{l}
\mu_S\left( [x_0-(q+d),0,x_2] \right) + \mu_S\left( [x_0-(q+d)-1,1,x_2-1] \right) \\
 - \mu_S\left( [x_0-1,1,x_2-1] \right).
\end{array} \\
\end{array}
\end{equation}

By applying the  recursive equality \eqref{mu1} to $\mu_S\left( [x_0-1,1,x_2-1] \right)$, we obtain 

\begin{equation}\label{mu4}
\mu_S\left( [x_0-1,1,x_2-1]\right) = \begin{array}[t]{l}
   \displaystyle\mu_S\left( [x_0-(q+d)-1,1,x_2-1] \right) \\[2ex]
 + \displaystyle\mu_S\left( [x_0-(q+d)-1,0,x_2-1]\right) \\[2ex]
 - \displaystyle\mu_S\left( [x_0-1,0,x_2-1]\right).
\end{array}
\end{equation}
Finally, by combining equations \eqref{mu3} and \eqref{mu4} we have
$$
\mu_S\left([x_0,0,x_2]\right) = \begin{array}[t]{l}
\mu_S\left([x_0-(q+d),0,x_2]\right) \\[1.5ex]
+ \mu_S\left([x_0-1,0,x_2-1]\right) \\[1.5ex]
- \mu_S\left([x_0-(q+d)-1,0,x_2-1]\right). \\
\end{array}
$$
This concludes the proof for $x_1=0$. The proof for the case $x_1=1$ is similar as the above case and it is left to the reader.
\end{proof}

%\begin{lem}
%For any $i\neq j$, $A_i \bigcap A_j = \emptyset$.
%\end{lem}

%\begin{proof}
%Let $m_1(q+d)+i\in A_i$ and $m_2(q+d)+j\in A_j$ such that
%$$
%m_1(q+d)+i=m_2(q+d)+j\in A_i\bigcap A_j.
%$$
%Since
%$$
%m_1(q+d)+i = m_2(q+d)+j \ \Longleftrightarrow \ (m_1-m_2)(q+d)+(i-j)=0,
%$$
%it follows that $i=j$ and $m_1=m_2$.
%\end{proof}

\begin{lem}\label{lem2} Let $i\in\{-1,0,1\}$. For all $x\in\Z\setminus\{i\}$,  $\m_{A_i}(x) = \m_{A_i}(x-(q+d))$.
\end{lem}

\begin{proof}
By definition of the sets $A_i$, for any integer $x\leq q+d+i-1$ such that $x\neq i$, we have
$$
\m_{A_i}(x) = \m_{A_i}(x-(q+d)) = 0.
$$
For any integer $x\geq q+d+i$, we obtain
$$
\begin{array}[t]{ll}
x\in A_i & \Longleftrightarrow \ \text{there exists}\ m\in\N \ \text{with}\ \ x=m(q+d)+i \\
&\Longleftrightarrow \ \text{there exists}\ m\in\N \ \text{with}\ \ x-(q+d)=(m-1)(q+d)+i \\
&\Longleftrightarrow \ x-(q+d)\in A_i.
\end{array}
$$
This completes the proof.
\end{proof}

\begin{lem}\label{lem3} Let $i\in\{-1,0,1\}$. For all $x\in\Z$, $\m_{B_i}(x) = \m_{C_i}(x-(2q+d))$.
\end{lem}

\begin{proof}
By definition of the multisets $B_i$ and $C_i$, for any integer $x\leq 2q+d+i-1$, we have
$$
\m_{B_i}(x) = \m_{C_i}(x-(2q+d)) = 0.
$$
For any integer $x\geq 2q+d+i$, we obtain
$$
\begin{array}[t]{ll}
x\in B_i & \Longleftrightarrow \ \text{there exists}\ m\in\N\setminus\{0,1\},\ 1\leq n\leq \left\lfloor m/2\right\rfloor \ \text{with}\ \ x=m(q+d)-nd+i \\
& \Longleftrightarrow \ \text{there exists}\ m\in\N\setminus\{0,1\},\ 1\leq n\leq \left\lfloor m/2\right\rfloor \ \text{with}\ \ x-(2q+d)=\begin{array}[t]{l}(m-2)(q+d)\\ -(n-1)d+i \end{array}\\
& \Longleftrightarrow \ \text{there exists}\ m\in\N,\ 0\leq n\leq \left\lfloor m/2\right\rfloor \ \text{with}\ \ x-(2q+d)=m(q+d)-nd+i \\
& \Longleftrightarrow \ x-(2q+d)\in C_i.
\end{array}
$$
This completes the proof.
\end{proof}

%\begin{lem}
%For any positive integer $x$ such that $x\geq q+d$,
%$$
%x\in \overline{B_i} \ \Longrightarrow\ x-(q+d) \in \overline{B_i}.
%$$
%\end{lem}

%\begin{proof}
%$$
%x-(q+d)\in B_i \begin{array}[t]{l}
%\Longleftrightarrow \ \exists m\in\N,\ 1\leq n\leq \left\lfloor m/2\right\rfloor \ \ \text{s.t.}\ \ x-(q+d)=m(q+d)-nd+i \\
%\Longleftrightarrow \ \exists m\in\N,\ 1\leq n\leq \left\lfloor m/2\right\rfloor \ \ \text{s.t.}\ \ x=(m+1)(q+d)-nd+i \\
%\Longrightarrow \ \exists m\in\N,\ 1\leq n\leq \left\lfloor m/2\right\rfloor \ \ \text{s.t.}\ \ x=m(q+d)-nd+i \\
%\Longrightarrow \ x\in B_i.
%\end{array}
%$$
%\end{proof}

We may now prove Theorem \ref{thm1}.

{\it Proof of Theorem \ref{thm1}}. By double induction on $x_0$ and $x_2$.
\par For $x_0<0$, since $[x_0,x_1,x_2]\notin S$ and $C_i\cap (\Z\setminus\N) = \emptyset$ for all $i\in\{-1,0,1\}$, it follows that
$$
\mu_S([x_0,x_1,0]) = {(-1)}^{x_1}(\m_{A_0}-\m_{A_1}+2\m_{B_0}-\m_{B_{-1}}-\m_{B_{1}})(x_0) = 0,
$$
and
$$
\mu_S([x_0,x_1,x_2]) = {(-1)}^{x_1}(2\m_{C_0}-\m_{C_{-1}}-\m_{C_{1}})(x_0-x_2) = 0,
$$
for all $x_2\in\{1,\ldots,q-1\}$.
\par Now, for $x_0\in\N$, we suppose that the theorem is true for all values lesser than $x_0$ and all $x_2\in\{0,\ldots,q-1\}$. We distinguish different cases according to the values of $x_2$.

\par Case a) $x_2=0$. \\
For $x_0=0$, since $\mu_S([0,0,0])=\mu_S(0)=1$ and $\mu_S([0,1,0])=\mu_S(a+d)=-1$, it follows, by definition of the multisets $A_i$ and $B_i$, that
$$
\mu_S([0,x_1,0]) = {(-1)}^{x_1} = {(-1)}^{x_1}(\m_{A_0}-\m_{A_1}+2\m_{B_0}-\m_{B_{-1}}-\m_{B_1})(0).
$$
For $x_0=1$, since $\mu_S([1,0,0])=\mu_S(a)=-1$ and $\mu_S([1,1,0])=\mu_S(a+(a+d))=c_2(0,2a+d)-c_1(0,2a+d)=2-1=1$, it follows, by definition of the multisets $A_i$ and $B_i$, that
$$
\mu_S([1,x_1,0]) = {(-1)}^{x_1+1} = {(-1)}^{x_1}(\m_{A_0}-\m_{A_1}+2\m_{B_0}-\m_{B_{-1}}-\m_{B_1})(1).
$$
Suppose now that $x_0\geq 2$. From Proposition~\ref{prop1}, we have
$$\mu_S( [x_0,x_1,0] ) = \mu_S( [x_0-(q+d), x_1,0])  + \mu_S( [x_0-(q+d)-1,x_1,q-1] ) \\[1.5ex]
 - \mu_S( (x_0-2(q+d), x_1,q-1] ). $$
By induction hypothesis, we have
$$\mu_S( [x_0-(q+d),x_1,0] )\\ = {(-1)}^{x_1} \left( \m_{A_0}-\m_{A_1}+2\m_{B_0}-\m_{B_{-1}}-\m_{B_1}  \right)(x_0 - (q+d)),$$
$$\mu_S( [x_0-(q+d)-1,x_1, q-1])\\ = {(-1)}^{x_1} \left( 2\m_{C_0}-\m_{C_{-1}}-\m_{C_1}  \right)(x_0-(2q+d))$$
and
$$\mu_S( [x_0-2(q+d)-1,x_1,q-1] ) \\ = {(-1)}^{x_1} \left( 2\m_{C_0}-\m_{C_{-1}}-\m_{C_1}  \right)(x_0-(3q+2d)).$$
By Lemma~\ref{lem2}, since $x_0\geq2$, we already know that
$$
\m_{A_0}(x_0-(q+d)) = \m_{A_0}(x_0)\quad \text{and}\quad \m_{A_1}(x_0-(q+d)) = \m_{A_1}(x_0).
$$
Moreover, by Lemma~\ref{lem3}, we have
$$
\m_{C_i}(x_0-(2q+d)) = \m_{B_i}(x_0)\quad \text{and}\quad \m_{C_i}(x_0-(3q+2d)) = \m_{B_i}(x_0-(q+d)),
$$
for all $i\in\{-1,0,1\}$. Therefore,
$$
\mu_S( [x_0,x_1,0] ) \begin{array}[t]{l}
= {(-1)}^{x_1} \begin{array}[t]{l}
\left[ \left(\m_{A_0}-\m_{A_1}\right)(x_0-(q+d)) \right. \\
+ \left(2\m_{B_0}-\m_{B_{-1}}-\m_{B_1}\right)(x_0-(q+d)) \\
+ \left( 2\m_{C_0}-\m_{C_{-1}}-\m_{C_1}  \right)(x_0-(2q+d)) \\
\left. - \left( 2\m_{C_0}-\m_{C_{-1}}-\m_{C_1}  \right)(x_0-(3q+2d)) \right]
\end{array} \\
= {(-1)}^{x_1} \begin{array}[t]{l}
\left[ \left(\m_{A_0}-\m_{A_1}\right)(x_0) \right. \\
+ \left(2\m_{B_0}-\m_{B_{-1}}-\m_{B_1}\right)(x_0-(q+d)) \\
+ \left( 2\m_{B_0}-\m_{B_{-1}}-\m_{B_1}  \right)(x_0) \\
\left. - \left( 2\m_{B_0}-\m_{B_{-1}}-\m_{B_1}  \right)(x_0-(q+d)) \right]
\end{array} \\
= {(-1)}^{x_1} \left( \m_{A_0}-\m_{A_1}+2\m_{B_0}-\m_{B_{-1}}-\m_{B_1}  \right)(x_0).
\end{array}
$$

\par Case b) $x_2= 1$.\\ 
For $x_0=0$, since $\mu_S([0,0,1])=\mu_S(a+2d)=-1$ and $\mu_S([0,1,1])=\mu_S((a+d)+(a+2d))=c_2(0,2a+3d)-c_1(0,2a+3d)=2-1=1$, it follows, by definition of the multisets $C_i$, that
$$
\mu_S([0,x_1,1]) = {(-1)}^{x_1+1} = {(-1)}^{x_1}(2\m_{C_0}-\m_{C_{-1}}-\m_{C_1})(-1).
$$
For $x_0=1$, since $\mu_S([1,0,1])=\mu_S(a+(a+2d))=c_2(0,2a+2d)-c_1(0,2a+2d)=3-1=2$ and $\mu_S([1,1,1])=\mu_S(a+(a+d)+(a+2d))=-c_3(0,3a+3d)+c_2(0,3a+3d)-c_1(0,3a+3d)=-7+10-1=2$, it follows, by definition of the multisets $C_i$, that
$$
\mu_S([1,x_1,1]) = {(-1)}^{x_1}2 = {(-1)}^{x_1}(2\m_{C_0}-\m_{C_{-1}}-\m_{C_1})(0).
$$
Suppose now that $x_0\geq 2$. From Proposition~\ref{prop1}, we have
$$
\mu_S( [x_0, + x_1,1] ) =\mu_S( [x_0-(q+d), x_1,1] ) 
 + \mu_S( [x_0-1,x_1,0] ) - \mu_S( [x_0-(q+d)-1,x_1,0] ).
$$
By using Lemmas \ref{lem2} and \ref{lem3} and the induction hypothesis, we have
$$\mu_S( [x_0-(q+d),x_1,1])\\ = {(-1)}^{x_1} \left( 2\m_{C_0}-\m_{C_{-1}}-\m_{C_1}  \right)(x_0 - (q+d) - 1),$$
$$\mu_S( [x_0-1,x_1,0] )\\ = {(-1)}^{x_1} \left( \m_{A_0}-\m_{A_1}+2\m_{B_0}-\m_{B_{-1}}-\m_{B_1} \right)(x_0-1)$$
and
$$\mu_S( [x_0-(q+d)-1,x_1,0) \\ = {(-1)}^{x_1} \left( \m_{A_0}-\m_{A_1}+2\m_{B_0}-\m_{B_{-1}}-\m_{B_1} \right)(x_0-(q+d)-1).$$

First, since the multiset difference $C_i\setminus B_i$ is equal to $A_i$ for all $i\in\{-1,0,1\}$, it follows that
$$
\m_{C_i}-\m_{B_i} = \m_{A_i}.
$$
Therefore,
$$
\mu_S( [x_0,x_1,1]) \begin{array}[t]{l}
= {(-1)}^{x_1} \begin{array}[t]{l}
\left[ \left(\m_{A_0}-\m_{A_1}\right)(x_0-1) - \left(\m_{A_0}-\m_{A_1}\right)(x_0-(q+d)-1) \right. \\
+ \left(2\m_{C_0}-\m_{C_{-1}}-\m_{C_1}\right)(x_0-(q+d)-1) \\
- \left( 2\m_{B_0}-\m_{B_{-1}}-\m_{B_1}  \right)(x_0-(q+d)-1) \\
\left. + \left( 2\m_{B_0}-\m_{B_{-1}}-\m_{B_1}  \right)(x_0-1) \right] 
\end{array} \\
= {(-1)}^{x_1} \begin{array}[t]{l}
\left[ \left(\m_{A_0}-\m_{A_1}\right)(x_0-1) - \left(\m_{A_0}-\m_{A_1}\right)(x_0-(q+d)-1) \right. \\
+ \left(2\m_{A_0}-\m_{A_{-1}}-\m_{A_1}\right)(x_0-(q+d)-1) \\
\left. + \left( 2\m_{B_0}-\m_{B_{-1}}-\m_{B_1}  \right)(x_0-1) \right]\\
\end{array} \\
= {(-1)}^{x_1} \begin{array}[t]{l}
\left[ \left(\m_{A_0}-\m_{A_1}\right)(x_0-1) + \left(\m_{A_0}-\m_{A_{-1}}\right)(x_0-(q+d)-1) \right. \\
\left. + \left( 2\m_{B_0}-\m_{B_{-1}}-\m_{B_1}  \right)(x_0-1) \right].
\end{array}
\end{array}
$$
Moreover, by Lemma~\ref{lem2}, since $x_0\geq2$, we know that
$$
\m_{A_0}(x_0-(q+d)-1) = \m_{A_0}(x_0-1)\quad \text{and}\quad \m_{A_{-1}}(x_0-(q+d)-1) = \m_{A_{-1}}(x_0-1).
$$
Finally, in this case, we obtain
$$
\mu_S( [x_0,x_1,1]) \begin{array}[t]{l}
= {(-1)}^{x_1} \begin{array}[t]{l}
\left[ \left(\m_{A_0}-\m_{A_1}\right)(x_0-1) + \left(\m_{A_0}-\m_{A_{-1}}\right)(x_0-(q+d)-1) \right. \\
\left. + \left( 2\m_{B_0}-\m_{B_{-1}}-\m_{B_1}  \right)(x_0-1) \right].
\end{array} \\
= {(-1)}^{x_1} \left[ \left(2\m_{A_0}-\m_{A_{-1}}-\m_{A_{1}}\right)(x_0-1) + \left( 2\m_{B_0}-\m_{B_{-1}}-\m_{B_1}  \right)(x_0-1) \right] \\
= {(-1)}^{x_1} \left( 2\m_{C_0}-\m_{C_{-1}}-\m_{C_1}  \right)(x_0-1).
\end{array}
$$

\par Case c) $x_2\geq 2$.\\ 
From Proposition~\ref{prop1}, we have
$$\mu_S( [x_0,x_1,x_2]) = \mu_S( [x_0-(q+d),x_1,x_2 ]) \\[1.5ex]
 + \mu_S( [x_0-1,x_1,x_2-1] )
 - \mu_S( [x_0-(q+d)-1,x_1,x_2-1] ).$$
By induction, we have
$$ \mu_S( [x_0-(q+d),x_1x_2] )\\ = {(-1)}^{x_1} \left( 2\m_{C_0}-\m_{C_{-1}}-\m_{C_1}  \right)(x_0 - x_2 - (q+d)),$$
$$\mu_S( [x_0-1,x_1,x_2-1] )\\ = {(-1)}^{x_1} \left( 2\m_{C_0}-\m_{C_{-1}}-\m_{C_1} \right)(x_0-x_2)$$
and
$$\mu_S( [x_0-(q+d)-1,x_1,x_2-1] ) \\ = {(-1)}^{x_1} \left( 2\m_{C_0}-\m_{C_{-1}}-\m_{C_1} \right)(x_0-x_2-(q+d)).$$
Therefore,
$$
\mu_S( [x_0,x_1 ,x_2]) \begin{array}[t]{l}
= {(-1)}^{x_1} \begin{array}[t]{l}
\left[ \left(2\m_{C_0}-\m_{C_{-1}}-\m_{C_1}\right)(x_0-x_2-(q+d)) \right. \\
+ \left( 2\m_{C_0}-\m_{C_{-1}}-\m_{C_1}  \right)(x_0-x_2) \\
\left. - \left( 2\m_{C_0}-\m_{C_{-1}}-\m_{C_1}  \right)(x_0-x_2-(q+d)) \right]
\end{array} \\
= {(-1)}^{x_1} \left( 2\m_{C_0}-\m_{C_{-1}}-\m_{C_1}  \right)(x_0-x_2).
\end{array}
$$
This completes the proof of Theorem~\ref{thm1}.
\littbox
\scriptsize
\begin{center}
\begin{longtable}{|r|rrrrrrrrrrr|}

\caption{ First values of $\mu_S([x_0,0,x_2])$ for $q=11$ and $d=5$.} \label{table} \\

\hline
\multicolumn{1}{|r|}{\backslashbox{$x_0$}{$x_2$}} & \multicolumn{1}{r}{$0$} & \multicolumn{1}{r}{$1$} & \multicolumn{1}{r}{$2$} & \multicolumn{1}{r}{$3$} & \multicolumn{1}{r}{$4$} & \multicolumn{1}{r}{$5$} & \multicolumn{1}{r}{$6$} & \multicolumn{1}{r}{$7$} & \multicolumn{1}{r}{$8$} & \multicolumn{1}{r}{$9$} & \multicolumn{1}{r|}{$10$} \\
\hline
\endfirsthead

\multicolumn{12}{c}%
{\tablename\ \thetable{} -- continued from previous page} \\
\hline
\multicolumn{1}{|r|}{\backslashbox{$x_0$}{$x_2$}} & \multicolumn{1}{r}{$0$} & \multicolumn{1}{r}{$1$} & \multicolumn{1}{r}{$2$} & \multicolumn{1}{r}{$3$} & \multicolumn{1}{r}{$4$} & \multicolumn{1}{r}{$5$} & \multicolumn{1}{r}{$6$} & \multicolumn{1}{r}{$7$} & \multicolumn{1}{r}{$8$} & \multicolumn{1}{r}{$9$} & \multicolumn{1}{r|}{$10$} \\
\hline\endhead

\hline \multicolumn{12}{|c|}{Continued on next page} \\ \hline
\endfoot

\hline 
\endlastfoot

$0$ & $\mathbf{1}$ & $\mathbf{-1}$ & $0$ & $0$ & $0$ & $0$ & $0$ & $0$ & $0$ & $0$ & $0$ \\
$1$ & $\mathbf{-1}$ & $\mathbf{2}$ & $\mathbf{-1}$ & $0$ & $0$ & $0$ & $0$ & $0$ & $0$ & $0$ & $0$ \\
$2$ & $0$ & $\mathbf{-1}$ & $\mathbf{2}$ & $\mathbf{-1}$ & $0$ & $0$ & $0$ & $0$ & $0$ & $0$ & $0$ \\
 & $0$ & $0$ & $\mathbf{-1}$ & $\mathbf{2}$ & $\mathbf{-1}$ & $0$ & $0$ & $0$ & $0$ & $0$ & $0$ \\
 & $0$ & $0$ & $0$ & $\mathbf{-1}$ & $\mathbf{2}$ & $\mathbf{-1}$ & $0$ & $0$ & $0$ & $0$ & $0$ \\
 & $0$ & $0$ & $0$ & $0$ & $\mathbf{-1}$ & $\mathbf{2}$ & $\mathbf{-1}$ & $0$ & $0$ & $0$ & $0$ \\
 & $0$ & $0$ & $0$ & $0$ & $0$ & $\mathbf{-1}$ & $\mathbf{2}$ & $\mathbf{-1}$ & $0$ & $0$ & $0$ \\
 & $0$ & $0$ & $0$ & $0$ & $0$ & $0$ & $\mathbf{-1}$ & $\mathbf{2}$ & $\mathbf{-1}$ & $0$ & $0$ \\
 & $0$ & $0$ & $0$ & $0$ & $0$ & $0$ & $0$ & $\mathbf{-1}$ & $\mathbf{2}$ & $\mathbf{-1}$ & $0$ \\
 & $0$ & $0$ & $0$ & $0$ & $0$ & $0$ & $0$ & $0$ & $\mathbf{-1}$ & $\mathbf{2}$ & $\mathbf{-1}$ \\
$q-1$ & $0$ & $0$ & $0$ & $0$ & $0$ & $0$ & $0$ & $0$ & $0$ & $\mathbf{-1}$ & $\mathbf{2}$ \\
$q$ & $0$ & $0$ & $0$ & $0$ & $0$ & $0$ & $0$ & $0$ & $0$ & $0$ & $\mathbf{-1}$ \\
$q+1$ & $0$ & $0$ & $0$ & $0$ & $0$ & $0$ & $0$ & $0$ & $0$ & $0$ & $0$ \\
 & $0$ & $0$ & $0$ & $0$ & $0$ & $0$ & $0$ & $0$ & $0$ & $0$ & $0$ \\
 & $0$ & $0$ & $0$ & $0$ & $0$ & $0$ & $0$ & $0$ & $0$ & $0$ & $0$ \\
$q+d-1$ & $\mathbf{0}$ & $0$ & $0$ & $0$ & $0$ & $0$ & $0$ & $0$ & $0$ & $0$ & $0$ \\

$q+d$ & $\mathbf{1}$ & $\mathbf{-1}$ & $0$ & $0$ & $0$ & $0$ & $0$ & $0$ & $0$ & $0$ & $0$ \\
$q+d+1$ & $\mathbf{-1}$ & $\mathbf{2}$ & $\mathbf{-1}$ & $0$ & $0$ & $0$ & $0$ & $0$ & $0$ & $0$ & $0$ \\
$q+d+2$ & $0$ & $\mathbf{-1}$ & $\mathbf{2}$ & $\mathbf{-1}$ & $0$ & $0$ & $0$ & $0$ & $0$ & $0$ & $0$ \\
 & $0$ & $0$ & $\mathbf{-1}$ & $\mathbf{2}$ & $\mathbf{-1}$ & $0$ & $0$ & $0$ & $0$ & $0$ & $0$ \\
 & $0$ & $0$ & $0$ & $\mathbf{-1}$ & $\mathbf{2}$ & $\mathbf{-1}$ & $0$ & $0$ & $0$ & $0$ & $0$ \\
 & $0$ & $0$ & $0$ & $0$ & $\mathbf{-1}$ & $\mathbf{2}$ & $\mathbf{-1}$ & $0$ & $0$ & $0$ & $0$ \\
 & $0$ & $0$ & $0$ & $0$ & $0$ & $\mathbf{-1}$ & $\mathbf{2}$ & $\mathbf{-1}$ & $0$ & $0$ & $0$ \\
 & $0$ & $0$ & $0$ & $0$ & $0$ & $0$ & $\mathbf{-1}$ & $\mathbf{2}$ & $\mathbf{-1}$ & $0$ & $0$ \\
 & $0$ & $0$ & $0$ & $0$ & $0$ & $0$ & $0$ & $\mathbf{-1}$ & $\mathbf{2}$ & $\mathbf{-1}$ & $0$ \\
 & $0$ & $0$ & $0$ & $0$ & $0$ & $0$ & $0$ & $0$ & $\mathbf{-1}$ & $\mathbf{2}$ & $\mathbf{-1}$ \\
$2q+d-1$ & $\mathbf{-1}$ & $0$ & $0$ & $0$ & $0$ & $0$ & $0$ & $0$ & $0$ & $\mathbf{-1}$ & $\mathbf{2}$ \\
$2q+d$ & $\mathbf{2}$ & $\mathbf{-1}$ & $0$ & $0$ & $0$ & $0$ & $0$ & $0$ & $0$ & $0$ & $\mathbf{-1}$ \\
$2q+d+1$ & $\mathbf{-1}$ & $\mathbf{2}$ & $\mathbf{-1}$ & $0$ & $0$ & $0$ & $0$ & $0$ & $0$ & $0$ & $0$ \\
 & $0$ & $\mathbf{-1}$ & $\mathbf{2}$ & $\mathbf{-1}$ & $0$ & $0$ & $0$ & $0$ & $0$ & $0$ & $0$ \\
 & $0$ & $0$ & $\mathbf{-1}$ & $\mathbf{2}$ & $\mathbf{-1}$ & $0$ & $0$ & $0$ & $0$ & $0$ & $0$ \\
$2q+2d-1$ & $\mathbf{0}$ & $0$ & $0$ & $\mathbf{-1}$ & $\mathbf{2}$ & $\mathbf{-1}$ & $0$ & $0$ & $0$ & $0$ & $0$ \\

$2q+2d$ & $\mathbf{1}$ & $\mathbf{-1}$ & $0$ & $0$ & $\mathbf{-1}$ & $\mathbf{2}$ & $\mathbf{-1}$ & $0$ & $0$ & $0$ & $0$ \\
$2q+2d+1$ & $\mathbf{-1}$ & $\mathbf{2}$ & $\mathbf{-1}$ & $0$ & $0$ & $\mathbf{-1}$ & $\mathbf{2}$ & $\mathbf{-1}$ & $0$ & $0$ & $0$ \\
$2q+2d+2$ & $0$ & $\mathbf{-1}$ & $\mathbf{2}$ & $\mathbf{-1}$ & $0$ & $0$ & $\mathbf{-1}$ & $\mathbf{2}$ & $\mathbf{-1}$ & $0$ & $0$ \\
 & $0$ & $0$ & $\mathbf{-1}$ & $\mathbf{2}$ & $\mathbf{-1}$ & $0$ & $0$ & $\mathbf{-1}$ & $\mathbf{2}$ & $\mathbf{-1}$ & $0$ \\
 & $0$ & $0$ & $0$ & $\mathbf{-1}$ & $\mathbf{2}$ & $\mathbf{-1}$ & $0$ & $0$ & $\mathbf{-1}$ & $\mathbf{2}$ & $\mathbf{-1}$ \\
 & $0$ & $0$ & $0$ & $0$ & $\mathbf{-1}$ & $\mathbf{2}$ & $\mathbf{-1}$ & $0$ & $0$ & $\mathbf{-1}$ & $\mathbf{2}$ \\
 & $0$ & $0$ & $0$ & $0$ & $0$ & $\mathbf{-1}$ & $\mathbf{2}$ & $\mathbf{-1}$ & $0$ & $0$ & $\mathbf{-1}$ \\
 & $0$ & $0$ & $0$ & $0$ & $0$ & $0$ & $\mathbf{-1}$ & $\mathbf{2}$ & $\mathbf{-1}$ & $0$ & $0$ \\
 & $0$ & $0$ & $0$ & $0$ & $0$ & $0$ & $0$ & $\mathbf{-1}$ & $\mathbf{2}$ & $\mathbf{-1}$ & $0$ \\
 & $0$ & $0$ & $0$ & $0$ & $0$ & $0$ & $0$ & $0$ & $\mathbf{-1}$ & $\mathbf{2}$ & $\mathbf{-1}$ \\
$3q+2d-1$ & $\mathbf{-1}$ & $0$ & $0$ & $0$ & $0$ & $0$ & $0$ & $0$ & $0$ & $\mathbf{-1}$ & $\mathbf{2}$ \\
$3q+2d$ & $\mathbf{2}$ & $\mathbf{-1}$ & $0$ & $0$ & $0$ & $0$ & $0$ & $0$ & $0$ & $0$ & $\mathbf{-1}$ \\
$3q+2d+1$ & $\mathbf{-1}$ & $\mathbf{2}$ & $\mathbf{-1}$ & $0$ & $0$ & $0$ & $0$ & $0$ & $0$ & $0$ & $0$ \\
 & $0$ & $\mathbf{-1}$ & $\mathbf{2}$ & $\mathbf{-1}$ & $0$ & $0$ & $0$ & $0$ & $0$ & $0$ & $0$ \\
 & $0$ & $0$ & $\mathbf{-1}$ & $\mathbf{2}$ & $\mathbf{-1}$ & $0$ & $0$ & $0$ & $0$ & $0$ & $0$ \\
$3q+3d-1$ & $\mathbf{0}$ & $0$ & $0$ & $\mathbf{-1}$ & $\mathbf{2}$ & $\mathbf{-1}$ & $0$ & $0$ & $0$ & $0$ & $0$ \\

$3q+3d$ & $\mathbf{1}$ & $\mathbf{-1}$ & $0$ & $0$ & $\mathbf{-1}$ & $\mathbf{2}$ & $\mathbf{-1}$ & $0$ & $0$ & $0$ & $0$ \\
$3q+3d+1$ & $\mathbf{-1}$ & $\mathbf{2}$ & $\mathbf{-1}$ & $0$ & $0$ & $\mathbf{-1}$ & $\mathbf{2}$ & $\mathbf{-1}$ & $0$ & $0$ & $0$ \\
$3q+3d+2$ & $0$ & $\mathbf{-1}$ & $\mathbf{2}$ & $\mathbf{-1}$ & $0$ & $0$ & $\mathbf{-1}$ & $\mathbf{2}$ & $\mathbf{-1}$ & $0$ & $0$ \\
 & $0$ & $0$ & $\mathbf{-1}$ & $\mathbf{2}$ & $\mathbf{-1}$ & $0$ & $0$ & $\mathbf{-1}$ & $\mathbf{2}$ & $\mathbf{-1}$ & $0$ \\
 & $0$ & $0$ & $0$ & $\mathbf{-1}$ & $\mathbf{2}$ & $\mathbf{-1}$ & $0$ & $0$ & $\mathbf{-1}$ & $\mathbf{2}$ & $\mathbf{-1}$ \\

\end{longtable}

\end{center}

\end{document}